\documentclass[12pt]{amsart}
\usepackage{amssymb}
\usepackage{amscd}
\usepackage{amsmath}
\usepackage[curve]{xypic}
\usepackage{hyperref}
\newtheorem{theorem}{Theorem}[section]
\newtheorem{corollary}[theorem]{Corollary}
\newtheorem{lemma}[theorem]{Lemma}

\newtheorem{definition}[theorem]{Definition}

\newtheorem{rem}[theorem]{Remark}
\begin{document}


\title[$G$-Fredholm Property]
{The $G$-Fredholm Property
of \\ the $\bar\partial$-Neumann Problem}
\author{Joe J Perez}
\maketitle
\section{Introduction}

Let $\mathcal H_1$ and $\mathcal H_2$ be Hilbert spaces and let $\mathcal B(\mathcal H_1,\mathcal H_2)$ be the space of bounded linear operators $A:\mathcal H_1\to\mathcal H_2$. An operator $A\in\mathcal B(\mathcal H_1,\mathcal H_2)$ is said to be {\it Fredholm} if first, the kernel of $A$ is finite-dimensional, and second the image of $A$ is closed and has finite codimension. An application of the open mapping theorem shows that the closedness requirement on the image is redundant. 
A well-known example of Fredholm operators (F. Riesz): if $C$ is a compact operator then ${\bf 1}-C$ is Fredholm. It is easy to see that the Fredholm property is equivalent to invertiblility modulo finite-rank operators or compact operators.

For a Fredholm operator $A$ its {\it index} is defined by

\[{\rm ind} (A) = \dim_{\mathbb C}(\ker A) - \dim_{\mathbb C}({\rm coker} A).\]

\noindent
The set of all Fredholm operators in $\mathcal B(\mathcal H_1,\mathcal H_2)$ is an open set in the norm topology of $\mathcal B(\mathcal H_1,\mathcal H_2)$ and the index is locally constant in this set. This means that the index is stable under perturbations that are small with respect to the operator norm. This stability suggests that it might be possible to calculate the index in some concrete analytic situation.

The main example of such a situation is given by elliptic differential operators acting in sections of vector bundles over compact manifolds. Choosing $\mathcal H_1$ and $\mathcal H_2$ to be appropriate Sobolev spaces of these sections, we find that elliptic operators are Fredholm. Probably the most important result concerning these elliptic operators is the Atiyah-Singer index theorem, which gives the index of the operator in terms of some characteristic classes involving its principal symbol, \cite {BGV}.

Let $M$ be a noncompact manifold (possibly with boundary) and $A$ an elliptic differential operator on $M$. Then $A$ is not necessarily Fredholm. That is, the kernel and/or cokernel of $A$ may be infinite-dimensional and/or the image of $A$ may not be closed. In particular, the index as defined above may not be well-defined, but there are notions generalizing the Fredholm property and the index. In this paper we will use one of these generalized Fredholm properties that makes sense when there is a free action of a unimodular Lie group $G$ on $M$ with quotient $X=M/G$, a compact manifold. Making appropriate choices of metric on $M$ and in the vector bundles over $M$ and using a Haar measure on $G$, we obtain Hilbert spaces of sections on which the $G$-action is unitary. This action allows us to define an trace ${\rm tr}_G$ in the algebra of operators commuting with the action of $G$. Restricting this trace to orthogonal projections $P_L$ onto $G$-invariant subspaces $L$ provides a dimension function $\dim_G$ given by 

\[\dim_G(L) = {\rm tr}_G(P_L).\]

Generalizing the previous definition, a $G$-invariant operator $A:\mathcal H_1\to\mathcal H_2$ is said to be $G$-{\it Fredholm} if $\dim_G\ker A<\infty$  and if there exists a closed, invariant subspace $Q\subset {\rm im}(A)$ so that $\dim_G(\mathcal H_2\ominus Q)<\infty$. With this definition, we prove the following:

\begin{theorem}  Let $M$ be a complex manifold with boundary which is strongly pseudoconvex. Let $G$ be a unimodular Lie group acting freely by holomorphic transformations on $M$ so that $M/G$ is compact. Then, for $q>0$, the Kohn Laplacian $\square$ in $L^2(M,\Lambda^{p,q})$ is $G$-Fredholm.\end{theorem}

\begin{corollary} If $M$ is as before and $q>0$, then the reduced Dolbeault cohomologies $H^{p,q}(M)$ have finite $G$-dimension.\end{corollary}

\begin{corollary}\label{penult} If $M$ is as before, let $L\subset(\ker\bar\partial)^\perp$ be closed and $G$-invariant. Then $\bar\partial|_L:L\to \overline{\bar\partial L}$ is $G$-Fredholm.\end{corollary}

\begin{rem}{\rm Examples of manifolds satisfying the hypotheses of the theorem are Grauert tubes of unimodular Lie groups. The unimodularity of $G$ is necessary for the definition of the $G$-Fredholm property.}\end{rem}

The $\bar\partial$-Neumann problem was proposed by Spencer in the 1950s as a method of obtaining existence theorems for holomorphic functions. Morrey in \cite{Mo} introduced the key {\it basic estimate} and the problem was solved by Kohn in \cite{K}. We use variants of the techniques in \cite{FK} in this work.

This generalized Fredholm property was first introduced in an abstract setting by M. Breuer \cite{B}.  In an analytical context, it was first used by L. Coburn, R. Moyer and I.M. Singer \cite{CMS} to define and calculate the real-valued index of elliptic almost-periodic pseudodifferential operators in $\mathbb R^{n}$.  Similarly, M. Atiyah \cite{A} defined and computed the real-valued index of elliptic operators on covering spaces of compact manifolds.  B. Fedosov and M. Shubin \cite{FS} working analytically (without Breuer's theory) defined and calculated the index of random elliptic operators in $\mathbb R^{n}$.  A. Connes and H. Moscovici \cite{CM} proved an $L^{2}$-index theorem for homogeneous spaces of noncompact Lie groups. Also, in \cite{S} M. Shubin used similar techniques to obtain an $L^2$-Riemann-Roch theorem for elliptic operators. In all this work, an important part of the analysis consists of showing that the operators under consideration have the property stressed previously: their images contain closed, invariant subspaces with finite codimension in an appropriate sense. In \cite{GHS}, $\Gamma$ is taken to be a discrete group and it is shown that the Kohn Laplacian $\square$ is $\Gamma$-Fredholm. Note that the natural boundary value problem for $\square$ (called the $\bar\partial$-Neumann problem) is not elliptic, but only subelliptic. In the present paper we extend this result from \cite{GHS} to the situation in which $G$ is a unimodular Lie group. When $G$ has a discrete cocompact subgroup $\Gamma \subset G$ the $\Gamma$-Fredholm property easily implies the $G$-Fredholm property.  Generically, however, it is not the case that a unimodular Lie group have such a subgroup, {\it cf.} \cite{M}. Using different methods, questions posed in \cite{GHS} have been answered and some results there strengthened in \cite{Br} and \cite{TCM}.

In section 2 we will introduce the $G$-trace for invariant operators in Hilbert $G$-modules.  Section 3 contains a description of abstract $G$-Fredholm operators and several useful properties. Section 4 treats the relevant results from the theory of the $\overline \partial$-Neumann problem.  In section 5 we discuss Hodge theory which links analytic results we obtain for $\square$ to the reduced $L^{2}$ Dolbeault cohomology of $M$.  In section 6 we prove that $\square$ is $G$-Fredholm and deduce the finite-dimensionality of the reduced Dolbeault cohomology for $q>0$. We also explore some easy consequences of the main theorem regarding the operator $\bar\partial$ on functions.


\section{Preliminaries}

 A {\it Hilbert $G$-module} is a Hilbert space with a (left) strongly continuous unitary action of $G$.  A {\it  free} Hilbert $G$-module is a Hilbert $G$-module which is unitarily and $G$-equivariantly isomorphic to the Hilbert space tensor product $L^{2}(G)\otimes \mathcal H$, where $\mathcal H$ is a Hilbert space and the Hilbert module structure is given by the  action of $G$ given by $s\mapsto R_{s}\otimes {\bf 1}_{\mathcal H}$,  where $R_s:L^2(G)\to L^2(G)$ is induced by the right translation $t\mapsto ts$ on $G$. A {\it projective} Hilbert $G$-module is a Hilbert $G$-module that can be embedded isometrically and $G$-equivariantly into a free Hilbert $G$-module.  Later on we will denote by $R_s$ the operator of the  action of $s\in G$ on arbitrary Hilbert $G$-modules. We will only need projective $G$-modules in this work, so from now on projective Hilbert $G$-modules will be called simply Hilbert $G$-modules.
    
  If there is an action of a group $G$ in a Hilbert space $\mathcal H$, denote the space of $G$-equivariant bounded linear operators in $\mathcal H$ by $\mathcal B(\mathcal H)^G$.  In other words $P\in \mathcal B(\mathcal H)^G$ if $P\in{\mathcal B}(\mathcal H)$  and  $R_{s}P=PR_{s}$ for every $s\in G$. An example of a projective Hilbert $G$-module is the image of a projection $P \in \mathcal B(L^2(M))^G$.

We will describe the Hilbert $G$-modules important to our discussion later, but first we restrict our attention to invariant operators on the group. Then we will define the $G$-invariant trace we actually need in the invariant operators on $L^2(M)$.   For any $s\in G$ define left and right translations $L_s, R_s:L^2(G)\to L^2(G)$ by
    $(L_{s}u)(t) = u(s^{-1}t)$, $(R_su)(t)=u(ts)$.  For $f \in L^{1}(G)$ and $u\in L^{2}(G)$, let  
 \[
 (L_{f}u)(t) = \int_{G}f(s)(L_{s}u)(t)ds = \int_{G}f(s)u(s^{-1}t)ds.
 \]
\noindent
The set $\{L_{f}\mid f\in L^{1}(G)\}$ forms an associative algebra of bounded operators 
in $L^2(G)$ which are right-invariant ({\it i.e.} commute with right translations).  
Define $\mathcal L_G\subset \mathcal B(L^{2}(G))$ to be the weak closure of this algebra.  Then $\mathcal L_{G}$ is a von Neumann algebra.  We will also need to consider operators $L_{f}$ for $f\in L^{2}(G)$.  These are defined on $C^{\infty}_{c}(G)$ and we may try to extend them  by continuity to $L^{2}(G)$. This is not always possible, but we  will be concerned only with those $L_{f}$ 
which are bounded, or, equivalently, can be extended to bounded linear operators in $L^{2}(G)$. 
The extended operator will be still denoted $L_f$ and it is then right-invariant  and belongs to
$\mathcal L_{G}$. It follows from the Schwartz kernel theorem that any bounded right-invariant operator in $L^2(G)$ can be presented in the form $L_f$ for a distribution $f$ on $G$.

%

We will need the following fact from about group von Neumann algebras ({\it cf.} \cite{P}, sections 5.1 and 7.2). There is a unique trace  ${\rm tr}_{G}$
on $\mathcal L_G \subset \mathcal B (L^2(G))$ agreeing with
\[ 
{\rm tr}_G (L_{f}^*L_{f}) = \int_G |f(s)|^2 ds, 
\]    
whenever $L_{f}\in \mathcal B (L^{2}(G))$ and $f\in L^2(G).$  Furthermore, 
${\rm tr}_G(A^* A)<\infty$ if and only if there is an $f \in L^2(G)$ for which 
$A= L_{f}\in \mathcal B(L^2(G))$.  
If we define $\tilde f(t) = \overline f(t^{-1})$, and  if $f_{k}, g_{k}\in L^2(G)$, $k=1,\dots, N$, 
then the operator $L_{h}= \sum_{1}^N L_{\tilde f_k} L_{g_k}$ is in ${\rm Dom}({\rm tr}_G)$.  
Furthermore, $h$ is continuous and ${\rm tr}_G(L_{h}) = h(e)$.

\begin{rem}{\rm The unimodularity of the group is necessary for the trace property of} ${\rm tr}_G$. \end{rem}

 Now we bring our results on the group up to the manifold. Let $G$ be a Lie group and $G\to M \overset{p}\to X$ be a principal $G$-bundle with compact base $X$.  In particular, this means that we have a free right action of 
    $G$ on $M$ with quotient space $X$, and $p:M\to X$ is the canonical projection.   Having a smooth free action 
 of $G$ on a manifold $M$ with a $G$-invariant measure $d{\bf x}$, 
 and fixing a Haar measure $dt$ on $G$, we obtain a natural quotient measure $dx$ on
 $X=M/G$ which allows us to present the Hilbert $G$-module $L^2(M)$ in the form
 \[
 L^{2}(M)\cong L^{2}(G)\otimes L^{2}(X),
 \]
 which makes it a free Hilbert $G$-module. It follows that we have
a decomposition of the von Neumann algebra of bounded invariant operators
 \[
 \mathcal B(L^{2}(M))^{G}\cong 
 \mathcal B(L^{2}(G))^{G}\otimes \mathcal B(L^{2}(X))\cong 
 \mathcal L_{G}\otimes \mathcal B(L^{2}(X)),
 \] 
where we have made the identification $\mathcal L_{G}\cong \mathcal B(L^{2}(G))^{G}$.  
In order to measure the invariant subspaces of $L^{2}(M)$, we need a trace on
$\mathcal L_{G}\otimes \mathcal B(L^{2}(X))$. It occurs that there exists a natural 
normal, faithful, semifinite trace on this algebra. It is denoted ${\rm Tr}_G$
and formally presented in the form
\[{\rm Tr}_G={\rm tr}_G\otimes {\rm Tr},\]
where $\rm Tr$ is the usual trace on $\mathcal B(L^{2}(X))$.
We describe the trace ${\rm Tr}_G$ in more detail. Let $(\psi_l)_{l\in\mathbb N}$ be an orthonormal basis for $L^2(X)$. Then
  \[
  L^2(M) \cong L^{2}(G)\otimes L^2(X)  
   \cong \bigoplus_{l\in \mathbb N} L^2(G)\otimes \psi_{l}.
   \]
    Denoting by $P_m$ the projection onto the $m^{th}$ summand, we obtain 
a matrix representation of $A\in \mathcal B(L^{2}(M))$ with elements $A_{lm} = P_l A P_m \in \mathcal B(L^{2}(G))$.  If $A\in \mathcal B(L^2(M))^G$, then these matrix elements are invariant operators in $L^{2}(G)$, and there exist distributions $h_{lm}$ on $G$ so that $A\in \mathcal B(L^{2}(M))^{G}$ has a matrix representation
    \begin{equation}\label{deco}
    A \leftrightarrow [A_{lm}]_{lm}=[L_{h_{lm}}]_{lm}.
    \end{equation}
    \begin{definition}   
    For positive $A\in  \mathcal B(L^{2}(M))^{G}$ define 
    \[{\rm Tr}_G (A) = \sum_{l\in \mathbb N} \ {\rm tr}_G (A_{ll}).
    \]
 \end{definition}

    The functional ${\rm Tr}_G$ is a normal, faithful, and semifinite trace and is independent of the basis $(\psi_{l})_{l}$ used in its construction, {\it cf.} Section V.2 of \cite{T}. We define the $G$-{\it Hilbert-Schmidt} operators
\begin{equation}\label{domtra}
{\rm Dom}_{1/2}({\rm Tr}_G)=\{A\in \mathcal B(L^{2}(M))^{G}\mid {\rm Tr}_G(A^{*}A)<\infty\}.
\end{equation}  
Also, define the $G$-{\it trace-class}, ${\rm Dom}({\rm Tr}_G)$, to be the vector space of finite linear combinations of the form $A^*B$, where $A,B\in {\rm Dom}_{1/2}({\rm Tr}_G)$.

\begin{rem}{\rm If $L$ is an arbitrary (projective) Hilbert $G$-module, then $L$ is the image of a $G$-invariant orthogonal projection $P$ in $L^{2}(G)\otimes \mathcal H$.  Thus the trace ${\rm Tr}_G$ on $L^{2}(G)\otimes \mathcal H$ restricts to one on $L$ defined by $A\mapsto {\rm Tr}_G(PAP)$.} \end{rem}

We will have to describe smoothness of functions, forms, and sections of vector bundles using $G$-invariant Sobolev spaces which we describe here. The $G$ action induces an invariant Riemannian metric on $M$ so that with respect to this structure $M$ has bounded geometry. As in \cite{G} and \cite{S1} we may construct appropriate partitions of unity and, with local geodesic coordinates, assemble $G$-invariant integer Sobolev spaces $H^s(M)$. If $E$ is a vector $G$-bundle over $M$, then we may introduce a $G$-invariant inner product structure on $E$. Together with the $G$-invariant measure on $M$ that we have described previously, we define the Hilbert spaces of sections of $E$ which we denote $H^s(M,E)$, for $s=0,1,2,\dots$. Because $X=M/G$ is compact, the spaces $H^s(M,E)$ do not depend on the choices of invariant metric on $M$ or of invariant inner product on $E$. Note that, in particular, spaces of sections in natural tensor bundles on a $G$-manifold have natural, invariant Sobolev structures.

\section{$G$-Fredholm Operators}

  We will explain and modify a generalized notion of the Fredholm property as introduced in \cite{B} in the setting of bounded operators in arbitrary von Neumann algebras.  By using the graph norm on the domain of the operator, it is easy to extend the results in \cite{B} to closed, densely defined operators as in \cite{S}. There, the von Neumann algebras in question were of invariant operators on Hilbert $\Gamma$-modules with $\Gamma$ a discrete group. Here we make the trivial extension to von Neumann algebras of invariant operators acting in Hillbert $G$-modules where $G$ is a unimodular Lie group rather than a discrete group. Also we will describe and utilize the property called $\Gamma$-density which was introduced and exploited in \cite{S}. A lemma regarding restrictions of Fredholm operators is proven here. 
  
  \begin{lemma}\cite{GHS}(2.1) Let $L$ be a Hilbert $G$-module and $L_{1}, L_{2}$ two Hilbert submodules of $L$ such that ${\rm dim}_{G}L_{1}>{\rm codim}_{G}L_{2}$ where the codimension means the dimension of the orthogonal complement of $L_{2}$ in $L$.  Then $L_{1}\cap L_{2}\neq \{0\}$ and ${\rm dim}_{G}L_{1}\cap L_{2} \ge {\rm dim}_{G}L_{1}-{\rm codim}_{G}L_{2}$.\end{lemma}
  
   \begin{definition} Let $L_0, \ L_1$ be Hilbert $G$-modules, $A:L_0 \to L_1$ a closed densely-defined linear operator commuting with the action of $G$.   Such an operator is called $G${\it-Fredholm} if the following conditions are satisfied:
\begin{itemize}
\item ${\rm dim}_G \ker  A< \infty $
\item there exists a $G$-invariant closed subspace $Q\subset L_1$ so that $Q \subset {\rm im}\ A$ and ${\rm codim}_G \ Q = {\rm dim}_G (L_1 \cap Q^{\perp}) < \infty .$
\end{itemize}  
\end{definition}
\begin{rem}{\rm Henceforth we will also use another notation: $L \ominus Q \stackrel {\rm def}= L\cap Q^{\perp}$.}  \end{rem}

\begin{definition}  Let $L$ be a Hilbert $G$-module and $Q \subset L$ a $G$-invariant subspace, not necessarily closed.  Then 
\begin{itemize}
\item If for every $\epsilon >0$ there is a $G$-invariant subspace $Q_{\epsilon} \subset Q$ such that $Q_{\epsilon}$ is closed in $L$ and ${\rm codim}_G Q_{\epsilon} < \epsilon $ in $L$, then $Q$ is called $G${\it -dense} in $L$.
\item $Q$ is called {\it almost closed} if $Q$ is $G$-dense in its closure $\overline Q$.
\end{itemize}
\end{definition}  

\begin{rem}{\rm It could happen that a $G$-invariant dense subspace $M\subset L$ in a Hilbert $G$-module $L$ not be $G$-dense.  For example, if $G$ is countable, then the space $Q$ of all functions on $G$ with finite support is not $G$-dense in $L^2(G)$.  Indeed, a closed subspace in $Q$ is necessarily finite-dimensional in the usual sense while any nontrivial closed invariant subspace in $L^2(G)$ must be infinite-dimensional.}\end{rem}

%
%
%
  



\begin{lemma}\label{ac}(Lemma 1.15 of \cite{S}) If $A:L_0 \to L_1$ is a $G$-Fredholm operator, then its image is almost closed. That is ${\rm im}(A)$ is $G$-dense in $\overline{{\rm im}(A)}$.\end{lemma}

\begin{corollary}(\cite{GHS} lemma 2.6)\label{fatcut}  Let $A:L_{0}\to L_{1}$ be a $G$-Fredholm operator and $L$ be a $G$-submodule of $L_{1}$ such that $L\subset \overline{{\rm im}(A)}$.  Then $L\cap {\rm im}(A)$ is $G$-dense in $L$.
\end{corollary}

\begin{lemma}\label{gdens}(\cite{S}, lemma 1.17)  Let $L$ be a Hilbert $G$-module, $L_{1}\subset L$, and $Q\subset L$ be $G$-invariant subspaces in $L$ so that $L_{1}$ is closed and $Q$ is $G$-dense in $L$.  Then $Q\cap L_{1}$ is $G$-dense in $L_{1}$.  More generally, if $Q$ is almost closed, then $Q\cap L_{1}$ is almost closed with closure equal $\overline Q\cap L_{1}$.\end{lemma}
%
%
\begin{lemma}\label{fredrest} If $A:\mathcal H_1\to \mathcal H_2$ is $G$-Fredholm and $L\hookrightarrow\mathcal H_1$ is closed and $G$-invariant, then $A|_L:L\to \overline{A(L)}$ is $G$-Fredholm.\end{lemma}
\begin{proof} This follows immediately from Lemma \ref{fatcut} and Lemma \ref{gdens}.
\end{proof}

\section{The $\overline \partial$-Neumann Problem}
  The principal references for this section are \cite{E,FK,GHS}. Let $M$ be a complex manifold with nonempty, smooth, strongly pseudoconvex boundary $bM$,  $\bar M=M\cup bM$, so that $M$ is the interior of $\bar M$, and ${\rm dim}_{\mathbb C}(M)=n$. For simplicity, let us also assume that $\bar M\subset\tilde M$, where $\tilde M$ is a complex neighborhood of $\bar M$ of the same dimension, such that $bM$ is in the interior of $\tilde M$. Let us choose a smooth function $\rho :\tilde M\to \mathbb R$ so that
 
  \[ M=\{z\mid \rho(z)<0\}, \ \ bM = \{z\mid \rho(z)=0\},\]
\noindent 
and for all $z\in bM$, we have $d\rho(z)\neq 0$.   
  
We describe the construction of $\square$ and its relevance to the solution of the $\bar\partial$-Neumann problem.  We seek a solution $u\in L^{2}(M)$ to the equation $\overline\partial u=\phi$ with $\phi \in L^{2}(M,\Lambda^{0,1})$, $\overline \partial \phi =0$.  Note that solutions will only be determined modulo the kernel of $\overline \partial$ consisting of all square-integrable holomorphic functions on $M$.  It is preferable to deal with self-adjoint operators, so since the Hilbert space adjoint $\overline \partial ^{*}$ of $\overline \partial$ satisfies $\overline{{\rm im} \ \overline \partial ^{*}}=(\ker \overline \partial)^{\perp}$, it is natural to seek $u$ of the form $u=\overline \partial^{*}v$, so that
     
      \begin{equation}\label{boh} \overline \partial\overline \partial^{*}v =\phi.  \end{equation}

Note that $\overline\partial\overline\partial^*$ is a self-adjoint operator.  In order to do away with the compatibility condition on $\phi$, let us add a term $\overline \partial^{*}\overline \partial v$, thus obtaining 

    \begin{equation}\label{KL}(\overline \partial\overline \partial^{*}+\overline \partial^{*}\overline \partial)v=\phi, \end{equation}

\noindent
where $\phi$ need not satisfy $\overline\partial \phi=0$. Notice that when $\overline \partial \phi =0$, \eqref{KL} reduces to (\ref{boh}) because applying $\overline \partial$ to \eqref{KL} gives $\overline \partial\overline \partial^{*}\overline \partial v =0$, which in turn implies 
\[0=\langle\overline \partial\overline \partial^{*}\overline \partial v, \overline \partial v \rangle= \|\overline \partial^{*}\overline \partial v\|_{L^{2}(M)}^{2}.\]
 Thus the new term in \eqref{KL} vanishes when the compatibility condition holds. Let us consider $\overline \partial $ as the maximal operator in $L^2(M)$ and let $\overline \partial^*$ be the Hilbert space adjoint operator. We will also use the corresponding Laplacian
\[\square=\square_{p,q}=\overline \partial\overline \partial^*+\overline \partial^* \overline \partial \quad \hbox{on} \quad
L^2(M,\Lambda^{p,q}).\]

We will denote the domain of any operator $A$ by ${\rm Dom}(A)$. The following lemma  gives a description of the operators $\overline \partial^*$, $\square$ as well as their domains ${\rm Dom}(\overline \partial^*)$, ${\rm Dom}(\square)$.
Let $\vartheta$ be the formal adjoint operator to $\overline \partial$, and let 
$\sigma=\sigma(\vartheta,\cdot)$ be its principal symbol.

\medskip
\begin{lemma}\label{dbar} \cite {GHS}  Let us assume that $M$ is strongly pseudoconvex.

\item{(i)}  The operator $\overline \partial^*$ can be obtained as the closure
of $\vartheta$ from the initial domain
\[  {\rm Dom}_0(\overline \partial^*)=\{\omega \mid \omega\in C_c^\infty(M, \Lambda^\bullet ),\ \sigma(\vartheta,d\rho)\omega=0 \ \ {\rm on}\ \ bM\}. \]
\item{(ii)}  The space ${\rm Dom}_0(\overline \partial^*)$ is dense in
${\rm Dom}(\overline \partial^*)\cap {\rm Dom}(\overline \partial)$ in the norm
\[(\| \omega\|_0^2 + \| \overline \partial^* \omega\|_0^2 + \|\overline \partial
\omega\|_0^2)^{1/2}, \ \omega\in {\rm Dom}(\overline \partial^*)\cap {\rm Dom}(\overline \partial).\]
\item{(iii)}  The operator $\square=\square_{p,q}$ can be obtained as the
closure of the operator $\overline \partial\vartheta + \vartheta\overline \partial$
from the initial domain
\[{\rm Dom}_0(\square) = \{\omega \mid  \omega,\ \overline \partial \omega, \vartheta\omega\in C^{\infty}(M, \Lambda^{\bullet})\cap L^2(M), \quad \sigma(\vartheta,d\rho)\omega=0,\]
\[ \sigma(\vartheta,d\rho)\overline \partial\omega=0 \ \
{\rm on}\  bM\}.\]
 For any $\omega\in {\rm Dom} (\square)\buildrel
\rm def \over = \{ \omega\in {\rm Dom}(\overline \partial)\cap {\rm Dom}(\overline \partial^*): \overline \partial
\omega\in {\rm Dom}(\overline \partial^*), \ \overline \partial^* \omega\in {\rm Dom}(\overline \partial)\}$ the following
integral identity holds
\[(\square\omega,\omega)=\| \overline \partial\omega\|_0^2 + \| \overline \partial^* \omega\|_0^2. \]
\end{lemma}
\noindent
The boundary conditions on $\omega$ are called the $\overline \partial${\it-Neumann conditions}.

  We describe the Friedrichs construction here for completeness \cite{FK}.  Suppose $\mathcal H$ is a Hilbert space and $Q$ is a Hermitian form defined on a dense subspace $\mathcal D\subset \mathcal H$ so that $Q(\phi,\phi) \ge \|\phi\|^2, \ \phi \in \mathcal D$.  Suppose further that $\mathcal D$ is a Hilbert space under the inner product $Q$.  Then there is a self-adjoint operator $F$ on $\mathcal H$ associated with $Q$:  For each $\alpha \in \mathcal H, \ \psi \mapsto \langle\alpha,\psi\rangle$ is a $Q$-bounded functional of $\psi \in \mathcal D$ since $|\langle\alpha,\psi\rangle|\le \|\alpha\|\|\psi\|\le\|\alpha\|\sqrt{Q(\psi,\psi)}$ .  By Riesz, we have a unique representative $\phi \in \mathcal D$ so that for all $\psi \in \mathcal D$, $Q(\phi,\psi)=\langle\alpha,\psi\rangle$.  Now define $T:\mathcal H\to \mathcal D\subset \mathcal H$ by $T\alpha = \phi$.  Then $\|T\alpha\|^2 \le Q(T\alpha,T\alpha) = \langle\alpha,T\alpha\rangle \le \|\alpha\|\|T\alpha\|$ so $T$ is a bounded operator.  Further, $T\alpha=0$ implies that $  \forall \psi \in \mathcal D,\ Q(T\alpha,\psi) = \langle\alpha,\psi\rangle=0$, hence $\alpha =0$ since $\mathcal D$ is assumed dense.  So $T$ is injective.  Now, $\langle T\alpha,\beta\rangle=\overline{\langle\beta,T\alpha\rangle}=\overline{Q(T\beta,T\alpha)}=Q(T\alpha,T\beta)=\langle\alpha,T\beta\rangle$.  Therefore $T$ is self-adjoint.  Put $F=T^{-1}$. The Friedrichs Extension theorem says that $F$ is the unique self-adjoint operator with ${\rm Dom}(F)\subset\mathcal D$ satisfying $Q(\phi,\psi) = \langle F\phi,\psi\rangle$ for all $\phi \in {\rm Dom}(F)$ and $\psi \in \mathcal D$.

\noindent
  In our case we will put $Q(\phi,\psi)=\langle\overline \partial \phi,\overline \partial \psi\rangle+\langle\vartheta \phi,\vartheta \psi\rangle+\langle\phi,\psi\rangle$ on the smooth forms satisfying the $\overline \partial$-Neumann boundary conditions.  Thus $F=\square + 1$.  
  
The following is a regularity result for $F$ and is the crux of the problem.
\begin{theorem}\label{globalize}
Let $M$ be strongly pseudoconvex, $U$ an open subset of $\bar M$ with compact clos ure, and $\zeta, \zeta_{1}\in C^{\infty}_{c}(U)$ for which $\zeta_{1}|_{{\rm supp}(\zeta)}=1$.  If $q>0$ and $\alpha|_{U}\in H^{s}(U,\Lambda^{p,q})$, then $\zeta(\square +1)^{-1}\alpha\in H^{s+1}(\bar M,\Lambda^{p,q})$ and there exist constants $C_s>0$ so that 
 \begin{equation}\label{prima}\|\zeta (\square +1)^{-1}\alpha\|_{s+1}^2\le C_s(\|\zeta_{1}\alpha\|_s^2+\|\alpha\|_0^2).\end{equation}
 \end{theorem}
 \begin{proof}This is Prop. 3.1.1 from \cite{FK} extended to the noncompact case in \cite{E}. \end{proof}
\begin{corollary}\label{bomb}Let $q>0$ and $\square=\int_{0}^{\infty}\lambda dE_{\lambda}$ be the spectral decomposition of the Laplacian in $L^{2}(M,\Lambda^{p,q})$.  If $\delta>0$ and $P=\int_{0}^{\delta}dE_{\lambda}$ then ${\rm im}(P)\subset C^{\infty}(\bar M,\Lambda^{p,q})$.\end{corollary}
\begin{proof}We show that ${\rm im}(P)\subset H^{s}_{\rm loc}(\bar M,\Lambda^{p,q})$ for all $s$.  Let $U, \zeta, \zeta_{1}$ be as in the previous theorem.  Since ${\rm im}((\square +1)^{-1})={\rm Dom}(\square)$ we have the following.  For every $u\in {\rm Dom}(\square)$ with $\square u +u\in H^{s}_{\rm loc}(\bar M)$, we have $u\in H^{s+1}_{\rm loc}(\bar M)$ and  
\[
 \|\zeta u\|_{s+1}^2\le C_s(\|\zeta_{1}(\square +1)u\|_s^2+\|(\square +1)u\|_0^2).
\]
\noindent
Let $u \in {\rm im}(P)$.  Applying the theorem with $s=0$, we have ${\rm im}(P)\subset H^{1}_{\rm loc}(M,\Lambda^{p,q})$.  Now assume $u\in {\rm im}(P)\subset H^{s-1}_{\rm loc}(\bar M,\Lambda^{p,q})$.  Then $(\square+1) u=(\square+1)Pu=P(\square+1)u\in H^{s-1}_{\rm loc}(\bar M,\Lambda^{p,q})$.  We conclude that $u\in H^{s}_{\rm loc}(\bar M,\Lambda^{p,q})$ and so ${\rm im}(P)\subset H^{s}_{\rm loc}(\bar M,\Lambda^{p,q})$.\end{proof}

 \noindent
 Here we repeat Theorem 2.2.9 of \cite{FK}, which gives interior Sobolev regularity of the Laplacian. Notice that it is local.
 
 \begin{lemma}Let $U, V$ be regions with $V\subset\bar V\subset U\subset\bar U\subset M$, and let $\zeta_1$ be a real $C^\infty$ function supported in $U$ with $\zeta_1=1$ on $V$. If $\phi\in {\rm Dom}(F)$ and $\zeta_1 F\phi\in H^s(M,\Lambda^{p,q})$ for some $s\ge 0$, then $\zeta\phi\in H^{s+2}(M,\Lambda^{p,q})$ for any real $\zeta\in\Lambda^{0,0}_0(V)$.  \end{lemma}

\begin{rem}{\rm We have that the images of the spectral projections of $\square$ corresponding to bounded intervals consist of forms that are smooth to the boundary, but we need that these forms belong to Sobolev spaces as in the interior as well. We cannot glue local estimates together as in \cite{GHS} because the last term in 

\[ \|\zeta u\|_{s+1}^2\le C_s(\|\zeta_{1}(\square +1)u\|_s^2+\|(\square +1)u\|_0^2)\]

\noindent
is not cut off and the proof uses crucially the compact support of the cutoff functions. To adjust for this, we will need to modify a number of claims from Sections 2.3 and 2.4 of \cite{FK}. As the proof is long and computationally detailed we will relegate it to an appendix and give the result here below, {\it cf.} Prop. 3.1.11 of \cite{FK}.

\begin{theorem} For every smooth $u\in {\rm Dom}(\square)\cap\Lambda^{p,q}$, we have 
\[\| u\|_{s+1}^2\lesssim \|\square u\|_s^2+ \| u\|^2\]
\noindent
for each positive integer $s$.
\end{theorem}

\begin{proof} From Lemma \ref{bigapriori} and again as in \cite{G} and \cite{S1} we may construct appropriate partitions of unity and glue together the local {\it a priori} estimates
\[\|\zeta u\|_{s+1}^2\lesssim \|\zeta_0F u\|_s^2+ \|\zeta_0 u\|^2 \quad (u\in {\rm Dom}(\square)\cap C^\infty)\]
to obtain the global estimate. 
\end{proof}

\begin{corollary}\label{bomb2}Let $q>0$ and $\square=\int_{0}^{\infty}\lambda dE_{\lambda}$ be the spectral decomposition of the Laplacian in $L^{2}(M,\Lambda^{p,q})$.  If $\delta>0$ and $P=\int_{0}^{\delta}dE_{\lambda}$ then ${\rm im}(P)\subset H^{\infty}(M,\Lambda^{p,q})$.\end{corollary}

\noindent
We need the following fact about Sobolev spaces on manifolds with boundary.
 
\begin{definition}For $s>0$, denote by $H^{-s}(\bar M)$ the dual space of $H^s(\bar M)$. {\it I.e.} $H^{-s}(\bar M)=(H^s(\bar M))'$.\end{definition}
 
\begin{lemma} Let $M$ be a manifold with boundary and $s>0$. Then $H^{-s}(\bar M)$ consists of elements of $H^{-s}(\tilde M)$ whose support is in $\bar M$.\end{lemma}
 
\begin{proof}See Remark 12.5 of \cite{LM}.\end{proof}
 
\begin{corollary}\label{applied}Let $q>0$ and $\square=\int_{0}^{\infty}\lambda dE_{\lambda}$ be the spectral decomposition of the Laplacian in $L^{2}(M,\Lambda^{p,q})$.  If $\delta>0$ and $P=\int_{0}^{\delta}dE_{\lambda}$ then $P: H^{-s}(\bar M,\Lambda^{p,q})\to H^{s}(M,\Lambda^{p,q})$ for any positive integer $s$.\end{corollary} 
 
\begin{proof} In Lemma \ref{bomb2} we established that spectral projections $P$ of $\square$ take $L^2(M)$ to $H^s(M)$ for all $s>0$. It follows that $P:H^{-s}(\bar M)\to L^2(M)$. Since $P^2=P$ on $H^\infty(M)\subset L^2(M)$, a dense subspace of all the $H^s(\bar M)$, $(s\in\mathbb R)$ we conclude that $P:H^{-s}(\bar M)\to H^s(M)$ for all $s>0$. \end{proof}
 
\section{Dolbeault-Hodge-Kodaira}

Let us describe the  reduced $L^2$ Dolbeault cohomology spaces on a
complex (generally non-compact) manifold $M$ with a given hermitian metric.
Denote the Hilbert space of all (measurable) square-integrable $(p,q)$-forms on $M$ by $L^2(M,\Lambda^{p,q})$.
The operator
\[\overline \partial :L^2(M,\Lambda^{p,q})\longrightarrow L^2(M,\Lambda^{p,q+1})\]
is defined as the maximal operator, {\it i.e.} its domain
$D^{p,q}=D^{p,q}(\overline \partial;M)$ is the set of all
$\omega \in L^2(M,\Lambda^{p,q})$ such that $\overline \partial \omega \in L^2(M,\Lambda^{p,q+1})$ where $\overline \partial $ is
applied in the sense of distributions. Obviously $\overline \partial^2=0$ on $D^{p,q}$
and we can form a complex
\[ L^2(M,\Lambda^{p,\bullet}):\quad 0\longrightarrow D^{p,0}\longrightarrow
D^{p,1}
\longrightarrow\dots\longrightarrow D^{p,n}\longrightarrow 0.
\]
The {\it reduced $L^2$-Dolbeault cohomology spaces of} $M$ are defined by:
\[L^2\bar H^{p,q}(M)=\ker (\overline \partial:D^{p,q}\to D^{p,q+1})/
\overline{{\rm im} \ (\overline \partial :D^{p,q-1}\to D^{p,q})}.\]
\noindent
Since $\ker\overline \partial $ is a closed subspace in $L^2$, the reduced cohomology space $L^2\bar H^{p,q}(M)$ is a Hilbert space. Note that the space $L^2\bar H^{0,0}(M)$ coincides with the space $L^2\mathcal O(M)$ of all square-integrable holomorphic functions on $M$.

\begin{lemma}\label{decomp} The following orthogonal decompositions hold:
\[ L^2(M,\Lambda^{\bullet})= \overline{{{\rm im} \ } \overline \partial} \oplus \ker \square \oplus \overline{{\rm im} \ \overline \partial^*} \qquad \ker   \overline \partial  = \overline{{\rm im} \  \overline \partial} \oplus \ker  \square.\]
In particular, we have an isomorphism of Hilbert $G$-modules
\begin{equation}L^2\bar H^{p,q}(M)=\ker  \square_{p,q}.\end{equation} \end{lemma}

\begin{corollary}${\rm im} \  \overline \partial \subset \overline{{\rm im} \  \square }$.\end{corollary}

 
\bigskip
 
             \section{The $G$-Fredholm Property of $\square$}

\noindent
We will need a description of $G$-operators in terms of their Schwartz kernels, {\it cf.} \eqref{domtra}.  If $P\in \mathcal B(L^{2}(M))^{G}$, its kernel $K_{P}$ satisfies
\[ K_P({\bf x},{\bf y})= K_P({\bf x}t,{\bf y}t),\quad t\in G.\]
\noindent
Thus $K_{P}$ descends to a distribution on the quotient $\frac{M\times M}{G}$.  The measure taken on $\frac{M\times M}{G}$ is simply the quotient measure.
 
  \begin{lemma} If $P: L^2(M) \to H^{\infty}(M)$ is a self-adjoint projection, then its Schwartz kernel $K_P$ is smooth. \end{lemma}
 
 \begin{proof} 
 

Since ${\bf y}\mapsto\delta_{\bf y}$ is a smooth function on $\bar M$ with values in $H^{-\infty}_{c}(\bar M)$, the composition
  \[({\bf x},{\bf y})\longmapsto (P\delta_{\bf y})({\bf x}) = \int_{M} K_P({\bf x},{\bf z})\delta_{\bf y}({\bf z})d{\bf z} = K_P({\bf x},{\bf y})\]
 \noindent 
 is jointly smooth.  \end{proof} 
 
 \begin{lemma} If $P\in \mathcal B(L^{2}(M))^{G}$ is a self-adjoint, invariant projection so that ${\rm im}(P)\subset C^{\infty}(M)$, then $K_{P}\in L^{2}(\frac{M\times M}{G})$.\end{lemma}

 \begin{proof}Fix ${\bf x}\in M$.  If $P:L^{2}(M)\to C^{\infty}(M)$, the closed graph theorem applied to $P$ implies $u\in L^{2}(M)\mapsto (Pu)({\bf x}) \in \mathbb C$ is a bounded linear functional.  The Riesz representation theorem then gives that there exists a function $h_{\bf x}\in L^{2}(M)$ so that
\[(Pu)({\bf x}) = \langle h_{\bf x},u\rangle \quad u\in L^{2}(M).\]
Since $(Pu)({\bf x})=\int_{M}K_{P}({\bf x},{\bf y})u({\bf y})d{\bf y}$, and agrees with $\langle h_{\bf x},u\rangle$ when $u$ has compact support, $h_{\bf x}=K_{P}({\bf x},\ \cdot \ )$ almost everywhere.  We conclude that for any ${\bf x}\in M$, $\int_{M}|K_{P}({\bf x},{\bf y})|^{2}d{\bf y}$ is finite.  


  Now consider $\phi({\bf x})=\int_{M}|K_{P}({\bf x},{\bf y})|^{2}d{\bf y}$.  The function $\phi$ is constant on orbits since the measure on $M$ is invariant;
\[\phi({\bf x}t)=\int_{M}|K_{P}({\bf x}t,{\bf y})|^{2}d{\bf y}=\int_{M}|K_{P}({\bf x},{\bf y}t^{-1})|^{2}d{\bf y}=\int_{M}|K_{P}({\bf x},{\bf y})|^{2}d{\bf y}=\phi({\bf x}).\]
\noindent
Thus $\phi$ descends to a function on $M/G=X$.  Since the map from $M$ to $C^{-\infty}_c(M)$ defined by ${\bf y}\mapsto \delta_{\bf y}$ is continuous, the composition
\[{\bf y} \mapsto P\delta_{\bf y}= K_P(\cdot, {\bf y})\]
is a continuous function $M\to L^2(M)$.  We may conclude that $\phi:X\to \mathbb R_{+}$ is continuous.  Denote by $\frac{d{\bf x}}{dt}$ the quotient measure on $X$.  The compactness of $X$ together with continuity of $\phi$ imply that $\int_{X}\phi({\bf x})\frac{d{\bf x}}{dt}<\infty$. Thus we have that $K_{P}\in L^{2}(\frac{M\times M}{G})$.
 \end{proof}

Choosing a measurable global section $x$ in $M$ and representing points ${\bf x}\in M$, ${\bf x}\to (t,x)\in G\times X$, we obtain an isomorphism of measure spaces $(M,d{\bf x})\cong (G\times X, dt\otimes dx)$.  Whenever $P\in \mathcal B(L^{2}(M))^{G}$ and $K_{P}\in L^{2}_{\rm loc}(M\times M)$, this isomorphism and the criterion for invariance allow a representation
\[K_P({\bf x},{\bf y})\longrightarrow K_{P}(t,x;s,y) \stackrel {\rm def} =\kappa(ts^{-1};x,y),\quad s,t\in G, \ x,y \in X\]
\noindent
with $\kappa \in L^{2}_{\rm loc}(G\times X \times X).$

 \begin{lemma}\label{normkappa} Let $P\in \mathcal B(L^{2}(M))^{G}$.  Then ${\rm Tr}_{G}(P^{*}P)= \int_{\frac{M\times M}{G}}|K_{P}|^{2}$.\end{lemma}
 \begin{proof} Let $(\psi_{k})_{k}$ be an orthonormal basis for $L^{2}(X)$.  In the decomposition $L^{2}(M)\cong \bigoplus_{k}L^{2}(G)\otimes \psi_{k}$, the invariant operator $P$ has a matrix representation $P\to [L_{h_{kl}}]_{kl}$.  In terms of this, we compute
\[{\rm Tr}_{G}(P^{*}P)=\sum_{l}{\rm tr}_{G}((P^{*}P)_{ll})=\sum_{l}{\rm tr}_{G}\left(\sum_{k}(P^{*})_{lk}P_{kl}\right)\]\[=\sum_{l}{\rm tr}_{G}\left(\sum_{k}P^{*}_{kl}P_{kl}\right)=\sum_{kl}{\rm tr}_{G}(L^{*}_{h_{kl}}L_{h_{kl}})=\sum_{kl}\|h_{kl}\|_{L^{2}(G)}^{2}\]
\noindent
by normality of ${\rm tr}_{G}$.  

Now, except on a set of measure zero, we have a description of $P$  
\[(Pu)({\bf x})=\int_{M}K_{P}({\bf x},{\bf y})u({\bf  y})d{\bf y}=(Pu)(t,x)=\int_{G\times X} ds dy\ \kappa(s;x,y)u(st,y).\]
Now, the distributional kernels $h_{ij}$ can be recovered from $\kappa$ by projecting into the summands in $L^{2}(M)\cong \bigoplus_{l} (L^{2}(G)\otimes \psi_{l})$,
\[h_{ij} = \int_{X\times X} \ dxdy\ \kappa(\ \cdot \ ;x,y)\psi_{j}(y)\overline\psi_{i}(x).\]
Let us compute the norm of $\kappa$ in $L^{2}(G\times X\times X)$.  Since $(\psi_{j})_{j}$ is an orthonormal basis for $L^{2}(X)$, the set $(\overline \psi_{i}\otimes \psi_{j})_{ij}$ forms an orthonormal basis for $L^{2}(X\times X)$.  By construction, $h_{ij}$ is equal the $ij^{th}$ Fourier coefficient of $\kappa$ with respect to the decomposition $L^{2}(G\times X\times X)\cong \bigoplus_{ij}(L^{2}(G)\otimes \psi_{i}\otimes \psi_{j})$.  Hence 
\[\sum_{ij} \|h_{ij}\|^{2}_{L^{2}(G)}=\|\kappa \|_{L^{2}(G\times X\times X)}^{2}.\]
\noindent
Thus ${\rm Tr}_{G}(P^{*}P)=\|\kappa \|_{L^{2}(G\times X\times X)}^{2} = \int_{\frac{M\times M}{G}}|K_{P}({\bf x},{\bf y})|^{2}\ \frac{d{\bf x}d{\bf y}}{dt}$.\end{proof}

 \begin{corollary}\label{theworks} If $P\in \mathcal B(L^{2}(M))^{G}$ is an invariant self-adjoint projection such that ${\rm im}(P)\subset H^{\infty}(M)$, then ${\rm Tr}_{G}(P)<\infty$.\end{corollary}

\begin{rem}{\rm All the previous results extend trivially to operators acting in bundles.} \end{rem}

\begin{theorem}For $q>0$, the operator $\square$ on $M$ is $G$-Fredholm.\end{theorem}
\begin{proof} Let $\square =\int_{0}^{\infty}\lambda dE_{\lambda}$ be the spectral decomposition of $\square$ and for $\delta >0$, $P=\int_{0}^{\delta}dE_{\lambda}$.  Thus ${\rm im}(1-P)\subset {\rm im}(\square)$.  Further, ${\rm im}(P)\subset L^{2}(M,\Lambda^{p,q})$ is closed, invariant and, by Corollary \ref{bomb}, ${\rm im}(P)\subset C^{\infty}(M,\Lambda^{p,q})$.  Corollary \ref{theworks} implies that ${\rm codim}_{G}({\rm im}(1-P))<\infty$.  The requirement on the kernel of $\square$ is verified noting that $\ker(\square) \subset {\rm im}(P)$ in the above.\end{proof}

\begin{rem}{\rm By Theorem 5.4.9 of \cite{FK} and the discussion immediately following, one can deduce the same results for the boundary Laplacian $\square_b$.}\end{rem}

\begin{corollary} If $q>0$, ${\rm dim}_G \  L^{2}\bar H^{p,q}(M) < \infty$.                        
\end{corollary}
\begin{proof} By Lemma \ref{decomp} $L^{2}\bar H^{p,q}(M) = \ker(\square_{p,q}) = {\rm im}( E_{0})$ which has finite $G$-dimension.\end{proof}

\begin{corollary}\label{penult} For the operator $\overline\partial:L^2(M,\Lambda^{0,0})\to L^2(M,\Lambda^{0,1})$ we have that ${\rm im}(\overline\partial)$ is $G$-dense in $\overline{{\rm im}(\overline\partial)}$. Consequently, $\overline\partial:L^2(M,\Lambda^{0,0})$ restricted to $(\ker\bar\partial)^\perp$ is $G$-Fredholm.\end{corollary}

\begin{proof}By Lemma \ref{gdens} we have that ${\rm im}(\square)\cap \overline{{\rm im}(\overline\partial)}$ is $G$-dense in $\overline{{\rm im}(\overline\partial)}$.  The decomposition \eqref{decomp} implies that ${\rm im}(\square)\cap \overline{{\rm im}(\overline\partial)}\subset{\rm im}(\overline\partial)$.  Thus ${\rm im}(\overline\partial)$ is almost closed.\end{proof}

\begin{corollary}\label{dbarrest}If $L$ is a closed and invariant subspace of $(\ker\overline\partial_{0,0})^\perp$, then $\overline\partial |_L:L\to\overline{\overline\partial L}$ is $G$-Fredholm.\end{corollary}

\begin{proof} Apply Lemma \ref{fredrest} to Corollary \ref{penult}. \end{proof}

\begin{corollary} For any closed, invariant $L\subset L^2(M,\Lambda^{0,0})$, we have that $\bar\partial L$ is almost closed.
\end{corollary}
\begin{proof}Consider $L\cap (\ker\overline\partial_{0,0})^\perp$. Then $\bar\partial (L\cap (\ker\overline\partial_{0,0})^\perp)$ is $G$-dense in $\overline{\bar\partial L}$.\end{proof}

\section{Appendix}

Here we derive an {\it a priori} estimate for the Laplacian $\square$ by modifying some lemmata from \cite{FK}. To that end, we repeat some of their definitions.

\begin{definition}Denote by $\mathcal D^{p,q}$ the domain of the formal adjoint $\vartheta$ of $\bar\partial$ in $C_c^\infty(\bar M,\Lambda^{p,q})$. \end{definition}

\begin{definition} A special boundary chart $U$ is a chart intersecting $bM$ having the following properties: 
\begin{enumerate}
\item With $\rho$ the function defining $bM$, the functions $t_1,\dots,t_{2n-1},\rho$ form a coordinate system on $U$.
\item The coordinates $\{t_1,\dots,t_{2n-1}\}_{\rho=0}$ form a coordinate system on $bM\cap U$.
\item Having chosen a Riemannian structure in the cotangent bundle, we choose a local orthonormal basis $\omega_1,\dots,\omega_n$ for $\Lambda^{1,0}(\bar M)$ such that $\omega_n=\sqrt{2}\ \partial\rho$ on $U$.
\end{enumerate}\end{definition}

\noindent
With the tangential Fourier transform in a special boundary chart

\[\tilde u(\tau,\rho)=\frac{1}{(2\pi)^{(2n-1)/2}}\int_{\mathbb R^{2n-1}}e^{-i\langle t,\tau\rangle}u(t,\rho)dt,\]

\noindent
define for $s\in\mathbb R$, the operators 

\[\Lambda_{\bf t}^s u(t,\rho)=\frac{1}{(2\pi)^{(2n-1)/2}}\int_{\mathbb R^{2n-1}} e^{i\langle t,\tau\rangle}(1+|\tau|^2)^{s/2}\tilde u(\tau,\rho)d\rho d\tau\]

\noindent
(${\bf t}$ means tangential) and define the tangential Sobolev norms by

\[|||u|||_s^2 = \int_{\mathbb R^{2n-1}}\int_{-\infty}^0 (1+|\tau|^2)^s|\tilde u(\tau,\rho)|^2d\rho d\tau.\]

\noindent
With $D^j=D^j_{\bf t}=\frac{1}{i}\frac{\partial}{\partial t_j}$ for $j=1,\dots,2n-1$ the derivatives in tangential directions and $D^{2n} = D_\rho$, define the norms

\begin{equation}\label{tangnorms}|||Du|||_s^2 = \sum_1^{2n} |||D^j u|||_s^2 + |||u|||_s^2 \approx  |||u|||_{s+1}^2 + |||D_\rho u|||_s^2.\end{equation}

\noindent
In order to state the basic estimate, we need the quantity

\[E(u)^2 = \sum_{jk}\|\partial_{\bar z_k}u_j\|^2 + \int_{bM}|u|^2 +\|u\|^2.\] 

\noindent
\begin{definition}That the basic estimate is satisfied means that there exists a $C>0$ such that $E(u)^2\le C Q(u,u)$ uniformly for $u$ in $\mathcal D^{0,1}$. We will abbreviate this and similar estimates 

\[E(u)^2\lesssim Q(u,u)\qquad (u\in\mathcal D^{0,1}).\]
\end{definition}
\noindent
If $M$ is strongly pseudoconvex, then the basic estimate holds in $\mathcal D^{0,1}$ (Prop 2.1.4, \cite{FK}) and in fact in all $\mathcal D^{p,q}$ for which $q>0$ (Corollary 3.2.12, \cite{FK}).}\end{rem}

\noindent
We will systematically label sequences of real-valued, cutoff functions $(\zeta_k)_k\subset C^\infty_c(M)$ such that $\zeta_{k}|_{{\rm supp}(\zeta_{k+1})}=1$ for $k=0,1,2,\dots$.
\begin{lemma}\label{cutFolland} Let $U$ be a special boundary chart and let $\zeta,\zeta_0,\zeta_1$ be real-valued functions in $C^\infty_c(U)$ with $\zeta_1=1$ on ${\rm supp}(\zeta)$ and $\zeta_0=1$ on ${\rm supp}(\zeta_1)$. Then for $A=\zeta_1\Lambda_{\bf t}^k\zeta$ and for $A'$ the formal adjoint of $A$ with respect to the inner product on $L^2(M)$,

\[ Q(A u,A u)-\mathfrak{Re}\ Q( u,\zeta_0 A'A u) = \mathcal O(|||D\zeta_0 u|||_{k-1}^2)\]
\[Q(\zeta u,\zeta u)-\mathfrak{Re}\ Q( u,\zeta_0\zeta^2 u) = \mathcal O(\|\zeta_0 u\|^2),\]
\noindent
uniformly for $ u\in\mathcal D^{p,q}\cap\Lambda_0^{p,q}(U\cap\bar M)$.\end{lemma}
\begin{proof}These are simple consequences of the fact that the domain $\mathcal D^{p,q}$ of $\vartheta$ is preserved under the application of a cutoff function ({\it cf.} 2.3.2 of \cite{FK}) and lemmata 2.4.2 and 2.4.3 of \cite{FK} applied to $\zeta_0 u$. \end{proof}

\noindent
\begin{rem}{\rm If we assume further that $u\in {\rm Dom}(F)$, ({\it cf.} \cite{FK}, Prop. 1.3.5) we may write
\[ Q(A u,A u)-\mathfrak{Re}\ \langle\zeta_0F u,A'A u\rangle = \mathcal O(|||D\zeta_0 u|||_{k-1}^2)\]
\[ Q(\zeta u,\zeta u)-\mathfrak{Re}\ \langle\zeta_0F u,\zeta^2 u\rangle = \mathcal O(\|\zeta_0 u\|^2).\]

\noindent
It is in this localized form that 2.4.2, 2.4.3 of \cite{FK} will be useful in our Lemma \ref{biglem}, a substantial modification of Lemma 2.4.6 of \cite{FK}. We will need the following theorem (2.4.4 from \cite{FK}) unchanged. }\end{rem}

\begin{lemma}\label{2.4.4}For every $p\in bM$ there is a (small) special boundary chart $V$ containing $p$ such that $|||Du|||_{-1/2}^2\lesssim E(u)^2$ uniformly for $u\in\Lambda_0^{p,q}(V\cap\bar M)$.\end{lemma}

\noindent
The following is our local replacement of Lemma 2.4.6 of \cite{FK}.

\noindent
\begin{lemma}\label{biglem} Suppose the basic estimate holds in $\mathcal D^{p,q}$. Let $V$ be a special boundary chart in which the conclusions of Lemma \ref{2.4.4} hold, and let $\{\zeta_k\}_0^\infty$ be a sequence of real functions in $\Lambda^{0,0}_0(V\cap \bar M)$ such that $\zeta_k=1$ on ${\rm supp}\ \zeta_{k+1}$.  Then for each positive integer $k$, 
\begin{equation}\label{aag}|||D\zeta_k u|||^2_{(k-2)/2}\lesssim \|\zeta_0F u\|^2_{(k-2)/2}+\|\zeta_0 u\|^2.\end{equation}
\noindent
uniformly for $ u\in {\rm Dom}(F)\cap\mathcal D^{p,q}$.\end{lemma}
\begin{proof}Assuming the basic estimate, using Lemma \ref{2.4.4}, and noting that multiplication by $\zeta_1$ preserves $\mathcal D^{p,q}$, we have
\[|||D\zeta_1 u|||_{-1/2}^2\lesssim Q(\zeta_1 u,\zeta_1 u),\qquad  u\in\mathcal D^{p,q}\cap \Lambda_0^{p,q}(V\cap\bar M).\]


\noindent
If we insert a real-valued cutoff function $\zeta_0$ equal 1 on the support of $\zeta_1$ and apply Lemma \eqref{cutFolland}, to the form $\zeta_0 u$ we have 
\[|||D\zeta_1 u|||_{-1/2}^2\lesssim \mathfrak{Re}\ Q( u,\zeta_0\zeta_1^2 u)+\mathcal O(\|\zeta_0 u\|^2).\]
%
\[ =  \mathfrak{Re}\  \langle F u,\zeta_1^2 u\rangle+\mathcal O(\|\zeta_0 u\|^2) =  \mathfrak{Re}\  \langle\zeta_1 F u,\zeta_1 u\rangle+\mathcal O(\|\zeta_0 u\|^2).\]

\noindent
Now, by the generalized Schwartz inequality, we have
\[|||D\zeta_1 u|||_{-1/2}^2\lesssim  \mathfrak{Re}\  \langle \zeta_1F u,\zeta_1 u\rangle+\mathcal O(\|\zeta_0 u\|^2)\lesssim \|\zeta_1F u\|_{-1/2}\|\zeta_1 u\|_{1/2}+\mathcal O(\|\zeta_0 u\|^2).\]

\noindent
But for any $c>0$ there exists a $C>0$ sufficiently large so that
\[|||D\zeta_1 u|||_{-1/2}^2\lesssim C\|\zeta_1F u\|_{-1/2}^2+c\|\zeta_1 u\|_{1/2}^2+\mathcal O(\|\zeta_0 u\|^2).\]

\noindent
By the equivalence in \eqref{tangnorms}, $\|\zeta_1 u\|_{1/2}\le |||D\zeta_1 u|||_{-1/2}$, so
\[|||D\zeta_1 u|||_{-1/2}^2\lesssim \|\zeta_0F u\|_{-1/2}^2+\|\zeta_0 u\|^2,\]
\noindent
and we have shown that the lemma is true for $k=1$. Assume the lemma true for $k-1$ {\it i.e.} 
\begin{equation}\label{indhyp}|||D\zeta_{k-1} u |||^2_{(k-3)/2}\lesssim \|\zeta_0 F u \|^2_{(k-3)/2} +\|\zeta_0 u\|^2. \end{equation}
\noindent
We follow the proof of \cite{FK} 2.4.6, citing intermediate results. Abbreviating $\Lambda_{\bf t}^{(k-1)/2}=\Lambda$ and $A=\zeta_1\Lambda\zeta_k$,
\begin{equation}\label{commpush1}|||D\zeta_k u|||^2_{(k-2)/2}\lesssim |||D\zeta_1\Lambda\zeta_k u|||^2_{-1/2}+|||D\zeta_{k-1} u|||^2_{(k-3)/2}\end{equation}
\begin{equation}\label{commpush2}|||DAu|||^2_{-1/2}\lesssim C|||\zeta_1 F u|||^2_{(k-2)/2}+c|||D\zeta_{k} u|||^2_{(k-2)/2}+|||D\zeta_{k-1} u|||^2_{(k-3)/2}.\end{equation}
\noindent
Substituting \eqref{commpush2} into \eqref{commpush1} gives
\[|||D\zeta_k u|||^2_{(k-2)/2}\lesssim |||\zeta_1F u|||^2_{(k-2)/2}+|||D\zeta_{k-1} u|||^2_{(k-3)/2}.\]
\noindent
Using the inductive hypothesis \eqref{indhyp} yields
\[|||D\zeta_k u|||^2_{(k-2)/2}\lesssim |||\zeta_1F u|||^2_{(k-2)/2}+ \|\zeta_0F u\|_{(k-3)/2}^2+\|\zeta_0 u\|^2.\]
\noindent
Because of the support properties of the $\zeta_k$,
\[|||D\zeta_k u|||^2_{(k-2)/2}\lesssim |||\zeta_0F u|||^2_{(k-2)/2}+ \|\zeta_0F u\|_{(k-3)/2}^2+\|\zeta_0 u\|^2.\]
\noindent
This implies
\[|||D\zeta_k u|||^2_{(k-2)/2}\lesssim \|\zeta_0F u\|^2_{(k-2)/2}+\|\zeta_0 u\|^2\]
\noindent
for the following two reasons: First, 
\[|||\zeta_0F u|||^2_{(k-2)/2}\le\|\zeta_0F u\|^2_{(k-2)/2}\]
\noindent
since the latter differentiates in the normal direction and the former does not. Second, $\|\zeta_0F u\|^2_{(k-3)/2}\le\|\zeta_0F u\|^2_{(k-2)/2}$ obviously.\end{proof}
\begin{rem}{\rm
Lemma 2.4.6 needs real modification if we are to obtain a local statement; cutting off na\"ively: 
\[|||D\zeta_k  u|||^2_{(k-2)/2}\lesssim |||\zeta_1F u|||^2_{(k-2)/2}+\|\zeta_1F u\|^2\]
\noindent
is false! To see this, let $q$ be a function with small support near the origin and choose $\zeta_1$ so that $\zeta_1q=0$. Furthermore, let $u\in\ker(\square)^\perp$ solve $Fu=q$. Then the right-hand side of the inequality is zero while the left is not.

\medskip

\noindent
The following lemma corresponds to \cite{FK} (2.4.8). } \end{rem}

\begin{lemma}\label{bigapriori}Suppose the basic estimate holds in $\mathcal D^{p,q}$. Let $V$ be a special boundary chart on which the conclusions of Lemma \ref{2.4.4} hold. Let $U\subset\bar U\subset V$, and choose a real $\zeta_1\in\Lambda_0^{0,0}(V\cap\bar M)$ with $\zeta_1 = 1$ on $U$. Then for each real $\zeta\in\Lambda_0^{0,0}(V\cap\bar M)$, and each positive integer $s$.
\[\|\zeta u\|_{s+1}^2\lesssim \|\zeta_0F u\|_s^2+ \|\zeta_0 u\|^2, \]
\noindent
uniformly for $ u\in {\rm Dom}(F)\cap\mathcal D^{p,q}$.
\end{lemma}
\begin{proof} Induction on $s$: For $s=0$, set $\zeta=\zeta_2$ and apply the previous lemma with $k=2$ and $0=\zeta_3=\zeta_4=\dots$
\[\|\zeta u\|_1^2\lesssim \|D\zeta_1 u\|^2\lesssim \|\zeta_0 F u\|^2+\|\zeta_0 u\|^2. \]
\noindent
Now assume the claim true for $s-1$. Then 
\[\|\zeta u\|_{s+1}\lesssim\sum_{|\beta|=s+1}\|D^\beta\zeta u\|^2+\|\zeta u\|^2_{s}\]
\[\|\zeta u\|_{s+1}\lesssim\sum_{|\beta|=s+1}\|D^\beta\zeta u\|^2+\|\zeta_1F u\|^2_{s-1}+\|\zeta_0 u\|^2\]
\begin{equation}\label{sofar}\|\zeta u\|_{s+1}\lesssim\sum_{|\beta|=s+1}\|D^\beta\zeta u\|^2+\|\zeta_1F u\|^2_{s}+\|\zeta_0 u\|^2,\end{equation}
\noindent
so estimate $\|D^\beta\zeta u\|^2$ for $|\beta|=s+1$. Construct a sequence of cutoffs $\{\zeta_k\}_2^{2s+1}$ so that $\zeta=\zeta_{2s+2}$ and $\zeta_k=1$ on ${\rm supp}\ \zeta_{k+1}$. Then apply Lemma \ref{biglem} with $k=2s+2$ and $\zeta_j=0$ for $j>2s+2$. Then
\begin{equation}\label{cat}\|D_{\bf t}^\beta\zeta u\|^2\lesssim |||D\zeta u|||_{s}^2\lesssim\|\zeta_1 F u\|_s^2 + \|\zeta_0 u\|^2. \end{equation}
\noindent
Thus we got part of the first term on the right of \eqref{sofar} estimated by the latter terms. For $|\beta|=s$ we have
\begin{equation}\label{dog}\|D_{\bf t}^\beta D_\rho\zeta u\|^2\lesssim |||D\zeta u |||_s^2  \lesssim \|\zeta_1 F u\|_s^2 + \|\zeta_0 u\|^2. \end{equation}
\noindent
It remains to estimate $D_{\bf t}^\beta D_\rho^m\zeta u$ with $|\beta|+m=s+1$, $m\ge 2$. Follow FK back to p 34, equation (2.3.5). Here $F$ is written in terms of differentiation with respect to the coordinates of the special boundary chart:
\[F u = A_0 D^2_\rho  u + \sum_{j=1}^{2n-1}A_j D_{\bf t}^jD_\rho u+\sum_{jk=1}^{2n-1}A_{jk}D_{\bf t}^jD_{\bf t}^k u+B_0D_\rho u+\sum_{j=1}^{2n-1}B_jD_{\bf t}^j u+C u.\]
\noindent
Since $F$ is an elliptic operator, the matrices $A$ are invertible. Thus we may solve
\[  D^2_\rho  u = -A_0^{-1}\left[-F u +\sum_{j=1}^{2n-1}A_j D_{\bf t}^jD_\rho u+\sum_{jk=1}^{2n-1}A_{jk}D_{\bf t}^jD_{\bf t}^k u+B_0D_\rho u+\sum_{j=1}^{2n-1}B_jD_{\bf t}^j u+C u\right].\]
\noindent
Applying $\zeta D_{\bf t}^\beta D_\rho^m$ with $|\beta|+m=s+1$, $m\ge 2$ and inserting a cutoff $\zeta_1$, we obtain
\[ \zeta D_{\bf t}^\beta D^m_\rho  u = -\zeta D_{\bf t}^\beta D^{m-2}_\rho A_0^{-1}\left[-\zeta_1 F u +\sum_{j=1}^{2n-1}A_j D_{\bf t}^jD_\rho u \ + \right.\]
\[\left.+\sum_{jk=1}^{2n-1}A_{jk}D_{\bf t}^jD_{\bf t}^k u+B_0D_\rho u+\sum_{j=1}^{2n-1}B_jD_{\bf t}^j u+C u\right].\]
\noindent
As in Folland and Kohn, at this point an induction on $m$ (commuting the $\zeta$ through) gives that $\zeta D_{\bf t}^\beta D^m_\rho  u$ is expressed in terms of derivatives of $\zeta_1 F u$ of order $s-1$ $(=|\beta| + m-2)$ and derivatives of $\zeta u$ which have been previously estimated in \eqref{cat} and \eqref{dog}.
\end{proof}

\section{Acknowledgments}

I thank my advisor Mikhail A Shubin for posing this problem and for giving me much advice.  I am also very grateful to Miroslav Engli\v s, Gerald Folland, Emil Straube, and Alex Suciu for helpful discussions.

\end{document}